\documentclass[a4paper,10pt]{amsart}

\usepackage{amsmath,amsfonts,amssymb}
\usepackage{bm}
\newcounter{num}
\newcommand{\Rnum}[1]{\setcounter{num}{#1}\Roman{num}}

\numberwithin{equation}{section}

\usepackage{amsthm}
\usepackage{tikz}
\usetikzlibrary{arrows.meta} 

\newcommand{\Hom}{\mathrm{Hom}}
\newcommand{\Der}{\mathrm{Der}}

\theoremstyle{plain}
\newtheorem{thm}{Theorem}[section]
\newtheorem{lem}[thm]{Lemma}
\newtheorem{prop}[thm]{Proposition}
\newtheorem{cor}[thm]{Corollary}

\theoremstyle{definition}

\newtheorem{dfn}[thm]{Definition}
\newtheorem{exa}[thm]{Example}

\theoremstyle{remark}
\newtheorem{rem}[thm]{Remark}

\theoremstyle{plain}
\newtheorem*{thm*}{Theorem}
\newtheorem*{lem*}{Lemma}
\newtheorem*{prop*}{Proposition}
\newtheorem*{cor*}{Corollary}
\newtheorem*{conj*}{Conjecture}
\newtheorem*{mthm*}{Main Theorem}

\theoremstyle{definition}
\newtheorem*{ass*}{Assumption}
\newtheorem*{dfn*}{Definition}

\theoremstyle{remark}
\newtheorem*{rem*}{Remark}

\begin{document}

\title[Log FW-differentials and log regularity]{A criterion for log regularity via log Frobenius-Witt differentials}
\author{Ryoma Takeuchi}
\address{Department of Mathematics, Institute of Science Tokyo, 2-12-1 Ookayama, Meguro, Tokyo 152-8551, Japan}
\email{takeuchi.r.1627@m.isct.ac.jp}
\begin{abstract}
    In \cite{MR4412577}, T.\ Saito introduced FW-derivations and the modules of FW-differentials. He gave a regularity criterion in terms of the modules of FW-differentials.

    In this paper, we introduce logarithmic analogues of FW-derivations and the modules of FW-differentials. We study basic properties of them and give a logarithmic regularity criterion in terms of the modules of logarithmic FW-differentials.
\end{abstract}
\maketitle
\tableofcontents
\section{Introduction}

\indent

In his paper \cite{MR4412577}, T.\ Saito introduced \textit{Frobenius-Witt derivations} (or \textit{FW-derivations} for short). Here, we recall the formulation of FW-derivations.
Let $p$ be a prime number and $P(X,Y)\in \mathbb{Z}[X,Y]$ be the polynomial given by
\begin{equation*}
    P(X,Y)=\frac{(X+Y)^p-X^p-Y^p}{p}=\sum_{i=1}^{p-1}\frac{(p-1)!}{i!(p-i)!}X^i Y^{p-i}.
\end{equation*}
For a ring $A$ and an $A$-module $M$, a map $D\colon A\to M$ is said to be an \textit{FW-derivation} if the following conditions are satisfied: for any $a,b\in A$, we have
\begin{align*}
    D(a+b) &= D(a)+D(b)-P(a,b)D(p), \\
    D(ab) &= b^p D(a)+a^p D(b).
\end{align*}
Saito proved that every ring $A$ has a universal FW-derivation $w\colon A\to F\Omega_A$ (\cite[Lemma 2.1]{MR4412577}) and called the $A$-module $F\Omega_A$ \textit{the module of FW-differentials} of $A$ (\cite[Definition 2.2]{MR4412577}).
Recall that an $\mathbb{F}_p$-algebra $R$ is \textit{$F$-finite} if the Frobenius homomorphism on $R$ is finite.
By using the modules of FW-differentials, he gave a regularity criterion of the following form:
\begin{thm}[{\cite[Theorem 3.4]{MR4412577}}]\label{thm:Saito_reg}
    Let $R$ be a Noetherian local ring with the residue field $k$ of characteristic $p>0$. Assume that $k$ is $F$-finite and set $[k:k^p]=p^r$. Consider the following conditions:
    \begin{itemize}
        \item[$(1)$] $F\Omega_R$ is a free $R/pR$-module of rank $\dim(R)+r$.
        \item[$(2)$] $F\Omega_R\otimes_R k$ is a $k$-vector space of dimension $\dim(R)+r$.
        \item[$(3)$] $R$ is regular. 
    \end{itemize}
    Then we always have $(1)\Rightarrow (2)\Rightarrow (3)$. Moreover, if $R/\sqrt{pR}$ is essentially of finite type over an $F$-finite field $k_1$, then the three conditions $(1)$, $(2)$, and $(3)$ are equivalent.
\end{thm}

In \cite{MR4855866}, independently of the Saito's work \cite{MR4412577}, Hochster and Jeffries gave a regularity criterion via \textit{the universal perivation modules} $\tilde{\Omega}_{-\vert\mathbb{Z}}$ (which coincide with the modules $F\Omega_{-}$ of FW-differentials for $\mathbb{Z}_{(p)}$-algebras) under another assumption:
\begin{thm}[{\cite[Theorem 4.13]{MR4855866}}]\label{thm:HJ_reg}
    Let $V$ be a discrete valuation ring with uniformizer $p$ such that the residue field $V/pV$ is $F$-finite. Let $(R,\mathfrak{m},k)$ be a Noetherian local ring which is essentially of finite type over a power series ring over $V$ with finitely many variables. Set $[k:k^p]=p^r$.
    
    Then $R$ is regular if and only if the $R/pR$-module $\tilde{\Omega}_{R\vert\mathbb{Z}}$ is free of rank $\dim(R)+r$.
\end{thm}

\vspace{5pt}

In this paper, we introduce logarithmic analogues of FW-derivations and the modules of FW-differentials for prelogarithmic rings.
For a prelogarithmic ring $(R,Q,\alpha)$, we denote the module of logarithmic FW-differentials by $F\Omega_{(R,Q,\alpha)}$.
Moreover, we give a logarithmic regularity criterion of the following form:
\begin{mthm*}[{Theorem \ref{thm:logreg-cr}}]
    Let $(R,Q,\alpha)$ be a local prelog ring such that $R$ is a Noetherian local ring with the residue field $k$ of characteristic $p>0$, that $Q$ is integral, and that $\overline{Q}$ is fine and saturated. Assume that $k$ is $F$-finite and set $[k:k^p]=p^r$. Consider the following conditions:
    \begin{itemize}
        \item[\textup{(1)}] $F\Omega_{(R,Q,\alpha)}\otimes_R R/pR$ is a free $R/pR$-module of rank $\dim(R)+r$.
        \item[\textup{(2)}] $F\Omega_{(R,Q,\alpha)}\otimes_R k$ is a $k$-vector space of dimension $\dim(R)+r$.
        \item[\textup{(3)}] $(R,Q,\alpha)$ is log regular.
    \end{itemize}
    Then we always have $(1)\Rightarrow (2)\Leftrightarrow (3)$. Moreover, if $R/pR$ is $F$-finite, then the three conditions $(1)$, $(2)$, and $(3)$ are equivalent.
\end{mthm*}

Here, we give some comments on the proof of Main Theorem. First, we reduce to the case where $Q$ is sharp.
Next, we prove that the conditions (2) and (3) are equivalent by using \cite[Proposition 2.6]{MR4412577} mainly.
Finally, we prove $(3)\Rightarrow (1)$ by passing to completions.

\vspace{5pt}

We expect that, as in the non-logarithmic case, the modules of logarithmic FW-differentials have a relation with logarithmic cotangent complexes, which were established by Olsson and Gabber in \cite{MR2195148}.

\vspace{5pt}

The content of this paper is as follows. In Section 2, we recall notions and basic properties of monoids and (pre)logarithmic rings including logarithmic regularity.

In Section 3, we give the definition of \textit{logarithmic FW-derivations} for prelogarithmic rings (Definition \ref{def:logFW-der}) and prove that every prelogarithmic ring has a universal logarithmic FW-derivation (Theorem \ref{thm:logFW-exist}). Moreover, we study basic properties of the modules of logarithmic FW-differentials.

In Section 4, we prove Main Theorem (Theorem \ref{thm:logreg-cr}) after some lemmas.

\vspace{5pt}
\noindent
\textbf{Acknowledgement}.
The author would like to thank Yuri Yatagawa for valuable comments and discussions.
\section{Monoids and prelogarithmic rings}

\indent

In this section, we recall some basic notions of monoids and (pre)logarithmic rings.
In this paper, we assume that every ring (resp.\ monoid) is commutative with the unity and that every homomorphism of rings (resp.\ monoids) preserves the unity. For a monoid $Q$, we write the monoid operation on $Q$ multiplicatively and write $1$ for the unity of $Q$ if there is no confusion.
When we refer to a ring as a monoid, we always use the multiplicative structure.
\subsection{Monoids}

\indent

For a monoid $Q$, let $Q^\times$ and $Q^+$ denote the set of invertible elements of $Q$ and that of non-invertible elements of $Q$, respectively.
Let $\mathbf{Mon}$ denote the category of monoids.
For a monoid $Q$, let $Q^{gp}$ be the group of elements of the form $xy^{-1}$, where $x,y\in Q$.
Note that, for any abelian group $A$, there is a natural bijection
    \begin{equation*}
        \Hom_\mathbb{Z}(Q^{gp},A)\cong \Hom_\mathbf{Mon}(Q,A),
    \end{equation*}
where $\Hom_\mathbf{Mon}(Q,A)$ denotes the set of homomorphisms of monoids from $Q$ to $A$.
\begin{dfn}[{\cite[Chapter \Rnum{1}, Section 1.3]{MR3838359}}]
    Let $Q$ be a monoid.
    \begin{itemize}
        \item[(1)] $Q$ is called \textit{integral} if the canonical homomorphism $Q\to Q^{gp}$ is an injection.
        \item[(2)] $Q$ is called \textit{fine} if $Q$ is finitely generated and integral.
        \item[(3)] $Q$ is called \textit{saturated} if $Q$ is integral and $Q$ has the following property: if $x\in Q^{gp}$ and $x^n\in Q$ for some $n>0$, then $x\in Q$.
        \item[(4)] $Q$ is called \textit{sharp} if $Q^\times=1$.
    \end{itemize}
\end{dfn}

\begin{rem}\label{rem:mon}
    We give some remarks.
    \begin{itemize}
        \item[(1)] (\cite[Chapter \Rnum{1}, Proposition 1.3.3]{MR3838359}) Let $Q$ be an integral monoid. Then $\overline{Q}$ is also integral.
        \item[(2)] (\cite[Chapter \Rnum{1}, Proposition 1.3.5 (4)]{MR3838359}) Let $Q$ be an integral monoid. Then $Q$ is saturated if and only if $\overline{Q}$ is saturated.
        \item[(3)] Let $Q$ be a monoid. Then $Q$ is sharp if and only if the natural surjection $Q\to \overline{Q}$ is an isomorphism. Since $\overline{Q}$ is sharp, we have $\overline{\overline{Q}}\cong \overline{Q}$.
    \end{itemize}
\end{rem}

\begin{dfn}[{\cite[Chapter \Rnum{1}, Section 1.4]{MR3838359}}]
    Let $Q$ be a monoid.
    \begin{itemize}
        \item[(1)] A subset $I$ of $Q$ is called an \textit{ideal} if $ax\in I$ for any $a\in Q$ and any $x\in I$.
        \item[(2)] A \textit{prime ideal} of $Q$ is an ideal $\mathfrak{p}$ of $Q$ such that $Q\setminus \mathfrak{p}$ is a submonoid of $Q$.
        \item[(3)] The \textit{dimension} of $Q$ is the maximum length $d$ of a chain of prime ideals
        \begin{equation*}
            \mathfrak{p}_0\subset \cdots \subset \mathfrak{p}_d.
        \end{equation*}
        We denote it by $\mathrm{dim}(Q)$.
    \end{itemize}
\end{dfn}
Let $Q$ be a monoid. Note that $\emptyset\subset Q$ is the minimum prime ideal of $Q$ and $Q^+$ is the maximum prime ideal of $Q$.
We donote the set of prime ideals of a monoid $Q$ by $\mathrm{Spec}(Q)$.
\begin{lem}\label{lem:dim}
    Let $Q$ be a monoid. Then we have the equality $\dim(Q)=\dim(\overline{Q})$.    
\end{lem}
\begin{proof}
    Let $\pi$ denote the canonical surjection $Q\to \overline{Q}$. If $I$ is an ideal of $Q$, an element $x$ of $Q$ belongs to $I$ if and only if $\pi(x)$ belongs to the ideal $\pi(I)$. Therefore, for a prime ideal $\mathfrak{p}$ of $Q$, the ideal $\pi(\mathfrak{p})$ is a prime ideal of $\overline{Q}$.
    Moreover, we have $\pi^{-1}(\pi(\mathfrak{p}))=\mathfrak{p}$ for every $\mathfrak{p}\in \mathrm{Spec}(Q)$ and have $\pi(\pi^{-1}(\mathfrak{p}'))=\mathfrak{p}'$ for every $\mathfrak{p}'\in \mathrm{Spec}(\overline{Q})$. Thus we have an order-preserving one-to-one correspondence between $\mathrm{Spec}(Q)$ and $\mathrm{Spec}(\overline{Q})$, which gives the equality $\dim(Q)=\dim(\overline{Q})$.
\end{proof}

\begin{prop}\label{prop:mon1}
    Let $Q$ be a monoid.
    \begin{itemize}
        \item[\textup{(1)}] \textup{(\cite[Chapter \Rnum{1}, Proposition 1.3.5 (2)]{MR3838359})} If $Q$ is saturated, then $\overline{Q}^{gp}$ is a torsion-free abelian group.
        \item[\textup{(2)}] \textup{(\cite[Chapter \Rnum{1}, Proposition 1.4.7 (2)]{MR3838359})} If $Q$ is integral, then we have the inequality $\dim(Q)\leq \mathrm{rank}(\overline{Q}^{gp})$. Moreover, if $Q$ is fine, then the equality $\mathrm{dim}(Q)=\mathrm{rank}(\overline{Q}^{gp})$ holds.
    \end{itemize}
\end{prop}

\begin{prop}\label{prop:mon2}
    Let $Q$ be an integral monoid such that $\overline{Q}$ is fine and saturated. Then the canonical surjection $\pi\colon Q\to \overline{Q}$ admits a section.
\end{prop}

\begin{proof}
    Since $\overline{Q}$ is fine and saturated, the abelian group $\overline{Q}^{gp}$ is finitely generated and torsion-free by Proposition \ref{prop:mon1}. Thus $\overline{Q}^{gp}$ is free of finite rank and the canonical surjection $\pi\colon Q^{gp}\to \overline{Q}^{gp}$ admits a section $s\colon \overline{Q}^{gp}\to Q^{gp}$.
    Since $Q$ (resp.\ $\overline{Q}$) is integral, we may identify $Q$ (resp.\ $\overline{Q}$) as a submonoid of $Q^{gp}$ (resp.\ $\overline{Q}^{gp}$). Then it suffices to show that $s(\overline{Q})$ is contained in $Q$. Let $\xi$ be an element of $\overline{Q}$ and take $x\in Q$ such that $\pi(x)=\xi$. We write $s(\xi)=yz^{-1}$, where $y,z\in Q$. Then we have $\pi(y)=\pi(xz)$ in $\overline{Q}$ since $\pi s(\xi)=\pi(x)$.
    Thus there exists $u\in Q^\times$ such that $y=uxz$. Then we have $s(\xi)=yz^{-1}=ux\in Q$.
\end{proof}

\subsection{Prelogarithmic rings}

\begin{dfn}[{\cite[Chapter \Rnum{3}, Definition 1.1.1]{MR3838359}}]
    Let $R$ be a ring.
    \begin{itemize}
        \item[(1)] A \textit{prelogarithmic} structure (or a \textit{prelog} structure) on $R$ is a pair $(Q,\alpha)$ where $Q$ is a monoid and $\alpha\colon Q\to R$ is a homomorphism of monoids.
        \item[(2)] A \textit{logarithmic} structure (or a \textit{log} structure) on $R$ is a prelog structure $(Q,\alpha)$ on $R$ such that $\alpha^{-1}(R^\times)\cong R^\times$ via $\alpha$.
    \end{itemize}
\end{dfn}

A \textit{prelog ring} (resp.\ \textit{log ring}) is a ring $R$ endowed with a prelog structure $(Q,\alpha)$ on $R$ (resp.\ a log structure $(Q,\alpha)$ on $R$) which we denote by $(R,Q,\alpha)$.
A (pre)log ring $(R,Q,\alpha)$ is \textit{local} if $R$ is a local ring and that we have $\alpha^{-1}(R^\times)=Q^\times$.

Let $R$ be a ring. A homomorphism of (pre)log structures on $R$ from $(Q_1,\alpha_1)$ to $(Q_2,\alpha_2)$ is a homomorphism of monoids $g\colon Q_1\to Q_2$ satisfying $\alpha_2 g=\alpha_1$.
Let $\mathbf{pLog}_R$ (resp.\ $\mathbf{Log}_R$) denote the category of prelog structures on $R$ (resp.\ the category of log structures on $R$).
Let $(Q,\alpha)$ be a prelog structure on $R$. Define $Q^{log}$ to be the pushout of the following diagram of monoids:
\begin{equation*}
        \begin{tikzpicture}[auto,->]
            \node (a) at (0,2) {$\alpha^{-1}(R^\times)$};
            \node (b) at (0,0) {$R^\times$};
            \node (x) at (2,2) {$Q$};
            \draw[{Hooks[right]}->] (a) -- (x);
            \draw (a) -- node {$\scriptstyle \alpha$}(b);
        \end{tikzpicture}
\end{equation*}
Let $\iota_1$ and $\iota_2$ denote the natural homomorphism $Q\to Q^{log}$ and $R^\times \to Q^{log}$, respectively. Then the homomorphism $\alpha\colon Q\to R$ induces a homomorphism
\begin{equation*}
    \alpha^{log}\colon Q^{log}\to R.
\end{equation*}
We can check that $(Q^{log},\alpha^{log})$ defines a log structure on $R$ and that $(Q^{log},\alpha^{log})$ satisfies the following property: for every log structure $(Q',\alpha')$ on $R$ and every homomorphism $g\colon (Q,\alpha)\to (Q',\alpha')$, there is a unique homomorphism
\begin{equation*}
    g^{log}\colon (Q^{log},\alpha^{log})\to (Q',\alpha')
\end{equation*}
satisfying $g=g^{log}\iota_1$.
Therefore the functor $(Q,\alpha)\mapsto (Q^{log},\alpha^{log})$ is a left adjoint of the inclusion functor $\mathbf{Log}_R\to \mathbf{pLog}_R$. For a log ring $(R,Q,\alpha)$, set
\begin{equation*}
    (R,Q,\alpha)^{log}:=(R,Q^{log},\alpha^{log}).
\end{equation*}
We call the log ring $(R,Q,\alpha)^{log}$ \textit{the associated log ring} of the prelog ring $(R,Q,\alpha)$.
The categories $\mathbf{pLog}_R$ and $\mathbf{Log}_R$ has the initial objects $(1,1\to R)$ and $(R^\times,R^\times\hookrightarrow R)$, respectively. When we refer to a usual ring $R$ as a prelog ring (resp.\ a log ring), we assume that $R$ is equipped with the prelog structure $(1,1\to R)$ (resp.\ the log structure $(R^\times,R^\times\hookrightarrow R)$) unless otherwise stated.

Let $(R,Q,\alpha)$ and $(R',Q',\alpha')$ be (pre)log rings. A homomorphism of (pre)log rings from $(R,Q,\alpha)$ to $(R',Q',\alpha')$ is a pair of maps $(f,g)$, where $f\colon R\to R'$ is a homomorphism of rings and $g\colon Q\to Q'$ is a homomorphism of monoids satisfying $f\alpha=\alpha'g$.
Note that, for a prelog ring $(R,Q,\alpha)$, we have a natural homomorphism of prelog rings $(R,Q,\alpha)\to (R,Q,\alpha)^{log}$.

For a prelog ring $(R,Q,\alpha)$, let $I_\alpha$ denote the ideal of $R$ generated by $\alpha(Q^+)$.

\begin{exa}
    Let $R$ be a ring and $Q$ be a monoid. Then we have a prelog ring $(R[Q],Q,\iota)$, where $\iota$ denotes the natural homomorphism $Q\to R[Q]$. We also write $R[Q^+]$ for the ideal $I_\iota=\iota(Q^+)R[Q]$. If $Q$ is sharp, then we have an isomorphism
    \begin{equation*}
        R[Q]/R[Q^+]\cong R.
    \end{equation*}
\end{exa}

\begin{dfn}[{\cite[(2.1) Definition]{MR1296725}, \cite[Chapter \Rnum{3}, Section 1.11]{MR3838359}}]
    Let $(R,Q,\alpha)$ be a local prelog ring such that $R$ is Noetherian and that $\overline{Q}$ is fine and saturated. We say that $(R,Q,\alpha)$ is \textit{log regular} if the following conditions are satisfied:
    \begin{itemize}
        \item[(1)] $R/I_\alpha$ is a regular local ring.
        \item[(2)] The equality $\dim(R)=\dim(R/I_\alpha)+\dim(Q)$ holds.
    \end{itemize}
\end{dfn}

\begin{rem}
    Let $(R,Q,\alpha)$ be a local prelog ring. Then $(R,Q,\alpha)^{log}$ is also a local prelog ring and we have $\overline{Q}\cong \overline{Q^{log}}$. We verify this assertion.
    Let $\mathbf{Mon}^{\text{sh}}$ denote the full subcategory of $\mathbf{Mon}$ whose objects are sharp monoids. Then the functor $Q\mapsto \overline{Q}$ is a left adjoint of the inclusion functor $\mathbf{Mon}^{\text{sh}}\to \mathbf{Mon}$ and hence preserves colimits.
    Therefore we have the pushout diagram of sharp monoids:
    \begin{equation*}
        \begin{tikzpicture}[auto,->]
            \node (a) at (0,2) {$\overline{\alpha^{-1}(R^\times)}$};
            \node (b) at (0,0) {$\overline{R^\times}$};
            \node (x) at (2,2) {$\overline{Q}$};
            \node (y) at (2,0) {$\overline{Q^{log}}$};
            \draw (a) -- (x);
            \draw (a) -- (b);
            \draw (b) -- (y);
            \draw (x) -- (y);
        \end{tikzpicture}
    \end{equation*}
    Since $\alpha^{-1}(R^\times)=Q^\times$ and $R^\times$ are abelian groups, we have $\overline{\alpha^{-1}(R^\times)}=\overline{R^\times}=1$. Thus we have $\overline{Q}\cong \overline{Q^{log}}$.

    In addition, assume that $R$ is Noetherian and that $\overline{Q}$ is fine and saturated. Then $\overline{Q^{log}}$ is also fine and saturated.
    Moreover, the local prelog ring $(R,Q,\alpha)$ is log regular if and only if the associated local log ring $(R,Q,\alpha)^{log}$ is log regular. Indeed, we have $I_\alpha=I_{\alpha^{log}}$ by the construction of $\alpha^{log}$ and have
    \begin{equation*}
        \dim(Q)=\dim(\overline{Q})=\dim(\overline{Q^{log}})=\dim(Q^{log})
    \end{equation*}
    by using $\overline{Q}\cong \overline{Q^{log}}$ and Lemma \ref{lem:dim}.
\end{rem}

\begin{prop}\label{prop:log-reg1}
    Let $(R,Q,\alpha)$ be a local prelog ring such that $R$ is Noetherian.
    \begin{itemize}
        \item[\textup{(1)}] \textup{(\cite[Chapter \Rnum{3}, Proposition 1.10.17]{MR3838359})} Assume that $Q$ is fine and saturated and that $Q^{gp}$ is torsion-free. Then we have
            \begin{equation*}
                \dim(R)\leq \dim(R/I_\alpha)+\dim(Q).
            \end{equation*}
        \item[\textup{(2)}] Assume that $Q$ is integral and that $\overline{Q}$ is fine and saturated. Let $s\colon \overline{Q}\to Q$ be a section of the canonical surjection $Q\to \overline{Q}$, which exists by \textup{Proposition \ref{prop:mon2}}. Then $(R,\overline{Q},\alpha s)$ is also a local prelog ring. Moreover, $(R,Q,\alpha)$ is log regular if and only if $(R,\overline{Q},\alpha s)$ is log regular.
    \end{itemize}
\end{prop}

\begin{proof}
    We prove (2).
    Since $(R,Q,\alpha)$ is a local prelog ring, we have
    \begin{equation*}
        (\alpha s)^{-1}(R^\times)=s^{-1}(Q^\times)=\overline{Q}^\times.
    \end{equation*}
    This implies that $(R,\overline{Q},\alpha s)$ is a local prelog ring.
    For every $x\in Q$, there is a unique element $u_x$ of $Q^\times$ such that $x=u_x\cdot s\pi(x)$.
    By this factorization, we have $I_\alpha=I_{\alpha s}$. We have $\dim(Q)=\dim(\overline{Q})$ by Lemma \ref{lem:dim}. Thus $(R,Q,\alpha)$ is log regular if and only if $(R,\overline{Q},\alpha s)$ is log regular.
\end{proof}

Let $R$ be a ring, $Q$ be a fine sharp monoid, and $\iota\colon Q\to R[Q]$ be the natural homomorphism. We denote the $I_\iota$-adic completion of $R[Q]$ by $R[[Q]]$.

\begin{prop}
    Let $R$ be a ring, $Q$ be a fine sharp monoid, and $\iota\colon Q\to R[Q]$ be the natural homomorphism. We denote the composition $Q\stackrel{\iota}{\to} R[Q]\to R[[Q]]$ by $\widehat{\iota}$.
    \begin{itemize}
        \item[\textup{(1)}] \textup{(\cite[Chapter \Rnum{1}, Proposition 3.6.1 (4)]{MR3838359})} If $Q^{gp}$ is torsion free and $R$ is a domain, then $R[[Q]]$ is also a domain.
        \item[\textup{(2)}] \textup{(\cite[Chapter \Rnum{1}, Proposition 3.6.1 (5)]{MR3838359})} If $R$ is a local ring with maximal ideal $\mathfrak{m}$, then $(R[[Q]],Q,\widehat{\iota})$ is a local prelog ring and the maximal ideal of $R[[Q]]$ is generated by $\mathfrak{m}R[[Q]]$ and $I_\iota R[[Q]]$.
    \end{itemize}
\end{prop}

\begin{thm}[{\cite[Chapter \Rnum{3}, Theorem 1.11.2]{MR3838359}}]\label{thm:logreg-str}
    Let $(R,Q,\alpha)$ be a local prelog ring such that $R$ is Noetherian and $Q$ is fine, saturated, and sharp. Let $k=R/\mathfrak{m}$ denote the residue field of $R$.
    \begin{itemize}
        \item[\textup{(1)}] Suppose that $R$ is of equal characteristic. Then $(R,Q,\alpha)$ is log regular if and only if there exists a commutative diagram of the form
            \begin{equation*}
                \begin{tikzpicture}[auto]
                    \node (a) at (0,2) {$Q$};
                    \node (b) at (0,0) {$R$};
                    \node (x) at (2,2) {$k[[Q\oplus \mathbb{N}^r]]$};
                    \node (y) at (2,0) {$\widehat{R},$};
                    \draw[{Hooks[right]}->] (a) -- (x);
                    \draw[->] (b) -- (y);
                    \draw[->] (a) -- node {$\scriptstyle \alpha$}(b);
                    \draw[->] (x) -- node[swap] {$\scriptstyle \cong$}(y);
                    \draw (2.2,1) node {$\scriptstyle \varphi$};
                \end{tikzpicture}
            \end{equation*}
            where the top homomorphism is the natural injection and $\widehat{R}$ denotes the $\mathfrak{m}$-adic completion of $R$.
        \item[\textup{(2)}] Suppose that $R$ is of mixed characteristic. Let $p>0$ be the characteristic of $k$ and $V$ be a Cohen ring of $k$. Then $(R,Q,\alpha)$ is log regular if and only if there exists a commutative diagram of the form
            \begin{equation*}
                \begin{tikzpicture}[auto]
                    \node (a) at (0,2) {$Q$};
                    \node (b) at (0,0) {$R$};
                    \node (x) at (2,2) {$V[[Q\oplus \mathbb{N}^r]]$};
                    \node (y) at (2,0) {$\widehat{R},$};
                    \draw[{Hooks[right]}->] (a) -- (x);
                    \draw[->] (b) -- (y);
                    \draw[->] (a) -- node {$\scriptstyle \alpha$}(b);
                    \draw[->>] (x) -- node {$\scriptstyle \varphi$}(y);
                \end{tikzpicture}
            \end{equation*}
            where $\varphi$ is surjective and its kernel is generated by an element $f$ of $V[[Q\oplus \mathbb{N}^r]]$ whose constant term is $p$.
    \end{itemize}
\end{thm}

\begin{rem}
    Let $(R,Q,\alpha)$ be a local prelog ring such that $Q$ is fine, saturated, and sharp. If $(R,Q,\alpha)$ is log regular, then the numbers $r$ in (1) and (2) of Theorem \ref{thm:logreg-str} are equal to the dimension of $R/I_\alpha$. We prove this claim.

    Suppose that we have a diagram as in (1). Then we have $\widehat{R}/I_\alpha \widehat{R}\cong k[[\mathbb{N}^r]]$. Thus we have the equality
    \begin{equation*}
        \dim(R/I_\alpha)=\dim(\widehat{R}/I_\alpha \widehat{R})=\dim(k[[\mathbb{N}^r]])=r.
    \end{equation*}

    Suppose that we have a diagram as in (2). Then we have
    \begin{equation*}
        \widehat{R}/I_\alpha \widehat{R}\cong V[[\mathbb{N}^r]]/(\tilde{f}),
    \end{equation*}
    where $\tilde{f}$ is the image of $f$ in $V[[\mathbb{N}^r]]$. Since the constant term of $\tilde{f}$ is also $p$, the element $\tilde{f}$ is a non-zero-divisor of the regular local ring $V[[\mathbb{N}^r]]$.
    Thus we have the equality
    \begin{equation*}
        \dim(R/I_\alpha)=\dim(\widehat{R}/I_\alpha \widehat{R})=\dim(V[[\mathbb{N}^r]]/(\tilde{f}))=1+r-1=r.
    \end{equation*}
\end{rem}

\section{Logarithmic Frobenius-Witt differentials}

\indent

In this section, we introduce logarithmic FW-derivations and logarithmic FW-differentials.
Let $p$ be a prime number and $P(X,Y)\in \mathbb{Z}[X,Y]$ be the polynomial given by
\begin{equation*}
    P(X,Y)=\sum_{i=1}^{p-1} \frac{(p-1)!}{i!(p-i)!} X^i Y^{p-i}.
\end{equation*}
\begin{dfn}[{\cite[Definition 1.1]{MR4412577}}]
    Let $R$ be a ring and $M$ be an $R$-module. A map $D\colon R\to M$ is said to be a \textit{Frobenius-Witt derivation} or an \textit{FW-derivation} if the following condition is satisfied:
    for any $a,b\in R$, we have
    \begin{align*}
        D(a+b) &= D(a)+D(b)-P(a,b)D(p), \\
        D(ab) &= b^p D(a)+a^p D(b).
    \end{align*}
\end{dfn}

\begin{dfn}\label{def:logFW-der}
    Let $(R,Q,\alpha)$ be a prelog ring and $M$ be an $R$-module. A \textit{logarithmic FW-derivation} (or a \textit{log FW-derivation} for short) of $(R,Q,\alpha)$ with values in $M$ is a pair of maps $(D,\delta)$, where $D\colon R\to M$ is an FW-derivation and $\delta\colon Q\to M$ is a homomorphism of monoids satisfying
    \begin{equation*}
        D(\alpha(x))=\alpha(x)^p \delta(x)
    \end{equation*}
    for every $x\in Q$.
\end{dfn}

Let $F\mathrm{Der}_{(R,Q,\alpha)}(M)$ denote the set of log FW-derivations of $(R,Q,\alpha)$ with values in $M$. Then $F\Der_{(R,Q,\alpha)}(M)$ is a subset of the set
\begin{equation*}
    F\Der_R(M)\times \Hom_\mathbf{Mon}(Q,M).
\end{equation*}
Note that the set $\Hom_\mathbf{Mon}(Q,M)$ has a natural $R$-module structure and that there are natural isomorphisms of $R$-modules
\begin{equation*}
    \Hom_\mathbf{Mon}(Q,M)\cong \Hom_\mathbb{Z}(Q^{gp},M)\cong \Hom_R(R\otimes_\mathbb{Z} Q^{gp},M).
\end{equation*}

\begin{prop}\label{prop:logFW-fun}
    Let $(R,Q,\alpha)$ be a prelog ring.
    \begin{itemize}
        \item[\textup{(1)}] For an $R$-module $M$, the set $F\Der_{(R,Q,\alpha)}(M)$ admits a natural sub-$R$-module structure of the $R$-module $F\Der_R(M)\times \Hom_\mathbf{Mon}(Q,M)$.
        \item[\textup{(2)}] Let $u\colon M\to M'$ be a homomorphism of $R$-modules. Then the homomorphism of $R$-modules
            \begin{equation*}
                F\Der_R(M)\times \Hom_\mathbf{Mon}(Q,M)\to F\Der_R(M')\times \Hom_\mathbf{Mon}(Q,M')
            \end{equation*}
            sending $(D,g)$ to $(uD,ug)$ induces a homomorphism of $R$-modules
            \begin{equation*}
                F\Der_{(R,Q,\alpha)}(M)\to F\Der_{(R,Q,\alpha)}(M')
            \end{equation*}
            sending $(D,\delta)$ to $(uD,u\delta)$.
    \end{itemize}
\end{prop}
\begin{proof}
    (1) It is obvious that the pair $(R\stackrel{0}{\to} M,Q\stackrel{0}{\to} M)$ belongs to $F\Der_{(R,Q,\alpha)}(M)$. Let $(D,\delta)$ and $(D',\delta')$ be elements of $F\Der_{(R,Q,\alpha)}(M)$. For elements $c$ and $c'$ of $R$, put $D''=cD+c'D'$ and $\delta''=c\delta+c'\delta'$. Then $D''$ is an FW-derivation and $\delta''$ is a homomorphism of monoids. Moreover, for every $x\in Q$, we have
    \begin{align*}
        D''\alpha(x) &= cD\alpha(x)+c'D'\alpha(x) \\
        &= c\alpha(x)^p \delta(x)+c'\alpha(x)^p \delta'(x) \\
        &= \alpha(x)^p \delta''(x).
    \end{align*}
    Thus (1) is proved.

    (2) Let $(D,\delta)$ be an element of $F\Der_{(R,Q,\alpha)}(M)$. We show that $(uD,u\delta)$ belongs to $F\Der_{(R,Q,\alpha)}(M')$. For every $x\in Q$, we have
    \begin{align*}
        uD\alpha(x) &= u(D\alpha(x)) \\
        &= u(\alpha(x)^p \delta(x)) \\
        &= \alpha(x)^p u\delta(x).
    \end{align*}
    This implies that $(uD,u\delta)$ is an element of $F\Der_{(R,Q,\alpha)}(M')$.
\end{proof}

By Proposition \ref{prop:logFW-fun}, we have a functor $F\Der_{(R,Q,\alpha)}(-)\colon \mathbf{Mod}_R\to \mathbf{Mod}_R$ for a log ring $(R,Q,\alpha)$.

\begin{thm}\label{thm:logFW-exist}
    Let $(R,Q,\alpha)$ be a prelog ring. Then the functor $F\Der_{(R,Q,\alpha)}(-)$ is representable by a universal log FW-derivation
    \begin{equation*}
        (w\colon R\to F\Omega_{(R,Q,\alpha)}, w\log \colon Q\to F\Omega_{(R,Q,\alpha)}).
    \end{equation*}
\end{thm}

\begin{proof}
    Let $w_0\colon R\to F\Omega_R$ denote the universal FW-derivation of $R$.
    We define $F\Omega_{(R,Q,\alpha)}$ to be the quotient of the $R$-module
    \begin{equation*}
        F\Omega_R\oplus (R\otimes_\mathbb{Z} Q^{gp})
    \end{equation*}
    by the submodule generated by elements of the form
    \begin{equation*}
        (w_0\alpha(x),-\alpha(x)^p\otimes x),
    \end{equation*}
    for all $x\in Q$. Let $w\colon R\to F\Omega_{(R,Q,\alpha)}$ be the homomorphism sending $x$ to the class of $(w_0(x),0)$ and $w\log\colon Q\to F\Omega_{(R,Q,\alpha)}$ be the homomorphism sending $q$ to the class of $(0,1\otimes q)$.
    Then $(w,w\log)$ is a universal log FW-derivation. Indeed, for an $R$-module $M$, the homomorphism
    \begin{equation*}
        \Hom_R(F\Omega_{(R,Q,\alpha)},M)\to F\Der_{(R,Q,\alpha)}(M)
    \end{equation*}
    sending $u$ to $(uw,uw\log)$ is an isomorphism by the construction of $F\Omega_{(R,Q,\alpha)}$.
\end{proof}

\begin{dfn}
    Let $(R,Q,\alpha)$ be a prelog ring.
    \begin{itemize}
        \item[(1)] We call the $R$-module $F\Omega_{(R,Q,\alpha)}$ in Theorem \ref{thm:logFW-exist} \textit{the module of logarithmic FW-differentials} of $(R,Q,\alpha)$ (or \textit{the module of log FW-differentials} of $(R,Q,\alpha)$ for short).
        \item[(2)] We call the $R/pR$-module $\tilde{F}\Omega_{(R,Q,\alpha)}$ defined by
            \begin{equation*}
                \tilde{F}\Omega_{(R,Q,\alpha)}=F\Omega_{(R,Q,\alpha)}\otimes_R R/pR
            \end{equation*}
            \textit{the module of log FW-differentials modulo $p$}.
    \end{itemize}
\end{dfn}
By definition, the $R/pR$-module $\tilde{F}\Omega_{(R,Q,\alpha)}$ represents the functor
\begin{equation*}
    F\Der_{(R,Q,\alpha)}(-)\colon \mathbf{Mod}_{R/pR}\to \mathbf{Mod}_{R/pR}.
\end{equation*}
Similarly, we set $\tilde{F}\Omega_R=F\Omega_R\otimes_R R/pR$ for a ring $R$. By \cite[Corollary 2.4 (1)]{MR4412577}, we have $\tilde{F}\Omega_R=F\Omega_R$ for any $\mathbb{Z}_{(p)}$-algebra $R$.

\begin{rem}\label{rem:logFW}
    We give some remarks.
    \begin{itemize}
        \item[(1)] Let $R$ be a ring and $(1,\iota\colon 1\to R)$ be the initial prelog structure on $R$. Then we have $F\Der_{(R,1,\iota)}(-)\cong F\Der_R(-)$ and hence $F\Omega_{(R,1,\iota)}\cong F\Omega_R$.
        \item[(2)] Let $(R,Q,\alpha)$ be a prelog ring. By the construction of $F\Omega_{(R,Q,\alpha)}$, we see that the $R$-module $F\Omega_{(R,Q,\alpha)}$ is a quotient of $F\Omega_R\oplus(R\otimes_\mathbb{Z}Q^{gp})$. Similarly, the $R/pR$-module $\tilde{F}\Omega_{(R,Q,\alpha)}$ is a quotient of $\tilde{F}\Omega_R\oplus(R/pR\otimes_\mathbb{Z}Q^{gp})$.
    \end{itemize}
\end{rem}

\begin{prop}\label{prop:logFW-fin}
    Let $(R,Q,\alpha)$ be a prelog ring such that $R/pR$ is $F$-finite and that $Q$ is finitely generated. Then the $R/pR$-module $\tilde{F}\Omega_{(R,Q,\alpha)}$ is finitely generated.
\end{prop}
\begin{proof}
    As $R/pR$ is $F$-finite, the $R/pR$-module $\Omega_{R/pR}=\Omega_{(R/pR)/\mathbb{F}_p}$ is finitely generated.
    Hence the $R/pR$-module $\tilde{F}\Omega_R$ is finitely generated by \cite[Corollary 2.4 (2)]{MR4412577}. Then the assertion follows from Remark \ref{rem:logFW} (2).
\end{proof}

Let $(f,g)\colon (R,Q,\alpha)\to (R',Q',\alpha')$ be a homomorphism of prelog rings and
\begin{align*}
    & (w_R\colon R\to F\Omega_{(R,Q,\alpha)},w\log_R\colon Q\to F\Omega_{(R,Q,\alpha)}) \\
    & (w_{R'}\colon R'\to F\Omega_{(R',Q',\alpha')},w\log_{R'}\colon Q'\to F\Omega_{(R',Q',\alpha')})
\end{align*}
be the universal log FW-derivations of $(R,Q,\alpha)$ and $(R',Q',\alpha')$, respectively. Then the pair
\begin{equation*}
    (w_{R'}f\colon R\to F\Omega_{(R',Q',\alpha')},w\log_{R'}g\colon Q\to F\Omega_{(R',Q',\alpha')})
\end{equation*}
is a log FW-derivation of $(R,Q,\alpha)$ with values in $F\Omega_{(R',Q',\alpha')}$. By the universality, we have a homomorphism of $R$-modules
\begin{equation*}
    F\Omega_{(R,Q,\alpha)}\to F\Omega_{(R',Q',\alpha')}
\end{equation*}
and a homomorphism of $R'$-modules
\begin{equation*}
    F\Omega_{(R,Q,\alpha)}\otimes_R R'\to F\Omega_{(R',Q',\alpha')}
\end{equation*}
and a homomorphism of $R'/pR'$-modules
\begin{equation*}
    \tilde{F}\Omega_{(R,Q,\alpha)}\otimes_R R'\to \tilde{F}\Omega_{(R',Q',\alpha')}.
\end{equation*}
Note that the homomorphism $F\Omega_{(R,Q,\alpha)}\to F\Omega_{(R',Q',\alpha')}$ sends $w_R(a)$ to $w_Rf(a)$ and sends $w\log_R(x)$ to $w\log_{R'}g(x)$.

\begin{prop}\label{prop:logFW-basic}
    Let $(R,Q,\alpha)$ be a prelog ring. Then the following hold.
    \begin{itemize}
        \item[\textup{(1)}] The natural homomorphism $(R,Q,\alpha)\to (R,Q,\alpha)^{log}$ induces a natural isomorphism
            \begin{align*}
                F\Omega_{(R,Q,\alpha)} &\cong F\Omega_{(R,Q,\alpha)^{log}}.
            \end{align*}
        \item[\textup{(2)}] Let $I$ be an ideal of $R$ and $\pi\colon R\to R/I$ be the canonical surjection. Then the natural homomorphism
            \begin{equation*}
                F\Omega_{(R,Q,\alpha)}\otimes_R R/I\to F\Omega_{(R/I,Q,\pi\alpha)}
            \end{equation*}
            is surjective and its kernel is generated by the image of $w(I)\subset F\Omega_{(R,Q,\alpha)}$ in $F\Omega_{(R,Q,\alpha)}\otimes_R R/I$. In particular, we have a natural isomorphism
            \begin{equation*}
                F\Omega_{(R,Q,\alpha)}\otimes_R R/I_\alpha \cong F\Omega_{(R/I_\alpha,Q,\pi\alpha)}.
            \end{equation*}
        \item[\textup{(3)}] If $R$ is a ring over $\mathbb{F}_p$, then we have a natural isomorphism
            \begin{equation*}
                F\Omega_{(R,Q,\alpha)}\cong F^*\Omega_{(R,Q,\alpha)},
            \end{equation*}
            where $\Omega_{(R,Q,\alpha)}=\Omega_{(R,Q,\alpha)/\mathbb{F}_p}$ is the module of log differentials of the homomorphism of prelog rings $\mathbb{F}_p\to (R,Q,\alpha)$ and $F^*\Omega_{(R,Q,\alpha)}=\Omega_{(R,Q,\alpha)}\otimes_R R$ is the base change of $\Omega_{(R,Q,\alpha)}$ with respect to the Frobenius homomorphism $F\colon R\to R$.
    \end{itemize}
\end{prop}
\begin{proof}
    (1) Let $\iota_1$ and $\iota_2$ denote the natural homomorphisms $Q\to Q^{log}$ and $R^\times\to Q^{log}$, respectively.
    For an FW-derivation $D\colon R\to M$, we define the homomorphism of monoids $\varphi_D\colon R^\times \to M$ given by
    \begin{equation*}
        \varphi_D(u)=D(u)/u^p.
    \end{equation*}
    To prove the assertion, it suffices to show that the natural homomorphism
    \begin{equation}\label{eq:4.1}
        F\Der_{(R,Q,\alpha)^{log}}(M)\to F\Der_{(R,Q,\alpha)}(M)
    \end{equation}
    sending $(D,\delta')$ to $(D,\delta'\iota_1)$ is an isomorphism for every $R$-module $M$. First, we show that the map \eqref{eq:4.1} is an injection. Let $(D,\delta')$ be an element of $F\Der_{(R,Q,\alpha)^{log}}(M)$ and suppose $(D,\delta'\iota_1)=(0,0)$. For every $u\in R^\times$, we have
    \begin{equation*}
        \delta'\iota_2(u)=D\alpha^{log}(\iota_2(u))/\alpha^{log}(\iota_2(u))^p=0.
    \end{equation*}
    Since we have $\delta'\iota_1=\delta'\iota_2=0$, we also have $\delta'=0$ by the universality of $Q^{log}$. Therefore the map \eqref{eq:4.1} is an injection.
    
    Next, we show that the map \eqref{eq:4.1} is surjective. Let $(D,\delta)$ be an element of $F\Der_{(R,Q,\alpha)}(M)$. Then we have $\delta(x)=\varphi_D\alpha(x)$ for every $x\in \alpha^{-1}(R^\times)$. Thus, by the universality of $Q^{log}$, we have a homomorphism of monoids $\delta^{log}\colon Q^{log}\to M$ satisfying $\delta=\delta^{log}\iota_1$ and $\varphi_D=\delta^{log}\iota_2$. Then we only need to prove that $(D,\delta^{log})$ is an element of $F\Der_{(R,Q,\alpha)^{log}}(M)$. Let $\xi$ be an element of $Q^{log}$. Write $\xi=\iota_1(x)\cdot \iota_2(u)$ for some $x\in Q$ and $u\in R^\times$. Then we have
    \begin{align*}
        D\alpha^{log}(\xi) &= D(\alpha(x)\cdot u) \\
        &= u^p D\alpha(x)+\alpha(x)^p D(u) \\
        &= u^p \alpha(x)^p (\delta(x)+\varphi_D(u)) \\
        &= \alpha^{log}(\xi)^p \delta^{log}(\xi),
    \end{align*}
    which shows that $(D,\delta^{log})$ belongs to $F\Der_{(R,Q,\alpha)^{log}}(M)$.

    (2) For an $R/I$-module $M$, let $V(M)$ be the submodule of $F\Der_{(R,Q,\alpha)}(M)$ consisting of log FW-derivations $(D,\delta)$ satisfying $D(I)=0$. We show that the natural homomorphism
    \begin{equation}\label{eq:4.2}
        F\Der_{(R/I,Q,\pi\alpha)}(M)\to F\Der_{(R,Q,\alpha)}(M)
    \end{equation}
    sending $(\overline{D},\delta)$ to $(\overline{D}\pi,\delta)$ is injective and its image is equal to $V(M)$.
    If we have a natural isomorphism
    \begin{equation*}
        F\Der_{(R/I,Q,\pi\alpha)}(M)\cong V(M),
    \end{equation*}
    the $R/I$-module $F\Omega_{(R/I,Q,\pi\alpha)}$ is isomorphic to the quotient of $F\Omega_{(R,Q,\alpha)}\otimes_R R/I$ by the submodule generated by the image of $w(I)$.
    
    It is obvious that the image of this map is contained in $V(M)$. Since $\pi$ is surjective, the map \eqref{eq:4.2} is injective. Let $(D,\delta)$ be an element of $V(M)$. Then, for elements $a$ and $b$ of $R$ satisfying $a-b\in I$, we have
    \begin{equation*}
        D(a) = D(a-b)+D(b)-P(a-b,b)D(p) = D(b)
    \end{equation*}
    since we have $D(I)=0$ and $IM=0$.
    Thus $D$ induces an FW-derivation $\overline{D}\colon R/I\to M$ such that $D=\overline{D}\pi$, which shows that $V(M)$ is contained in the image of the map \eqref{eq:4.2}.

    For the latter assertion, assume $I=I_\alpha$. The image of $w(I)$ in $F\Omega_{(R,Q,\alpha)}\otimes_R R/I$ is equal to $0$ since we have $w\alpha(x)=\alpha(x)^pw\log(x)$ for every $x\in Q$. Therefore we have
    \begin{equation*}
        F\Omega_{(R,Q,\alpha)}\otimes_R R/I\cong F\Omega_{(R/I,Q,\pi\alpha)}.
    \end{equation*}

    (3) Let $F_*\colon \mathbf{Mod}_R\to \mathbf{Mod}_R$ denote the right adjoint functor of the base change functor $F^*\colon \mathbf{Mod}_R\to \mathbf{Mod}_R$ with respect to the Frobenius homomorphism $F\colon R\to R$. Since $p=0$ in $R$, giving an element of $F\Omega_{(R,Q,\alpha)}(M)$ is equivalent to giving a pair $(D\colon R\to M,\delta\colon Q\to M)$, where $D$ is a map and $\delta$ is a homomorphism of monoids satisfying
    \begin{align*}
        & D(a+b)=D(a)+D(b), \\
        & D(ab)=b^p D(a)+a^p D(b), \\
        & D\alpha(x)=\alpha(x)^p \delta(x)
    \end{align*}
    for every $a,b\in R$ and every $x\in Q$. Therefore the $R$-module $F\Der_{(R,Q,\alpha)}(M)$ is naturally isomorphic to the $R$-module $\Der_\theta(F_* M)$, where $\Der_\theta(F_* M)$ is the module of log derivations of the homomorphism $\theta \colon \mathbb{F}_p\to (R,Q,\alpha)$ with values in $F_* M$. Since we have
    \begin{equation*}
        \Der_\theta(F_* M)\cong \Hom_R(\Omega_{(R,Q,\alpha)},F_* M)\cong \Hom_R(F^*\Omega_{(R,Q,\alpha)},M),
    \end{equation*}
    the assertion follows.
\end{proof}

\begin{prop}
    Let $(R,Q,\alpha)$ be a prelog ring, $S$ be a submonoid of $Q$, and $T$ be a submonoid of $R$. Suppose that $\alpha(S)$ is contained in $T$.
    Then we have natural isomorphisms
    \begin{equation*}
        T^{-1}\tilde{F}\Omega_{(R,Q,\alpha)}\cong \tilde{F}\Omega_{(T^{-1}R,Q,\iota\alpha)}\cong \tilde{F}\Omega_{(T^{-1}R,S^{-1}Q,\tilde{\alpha})},
    \end{equation*}
    where $\iota\colon R\to T^{-1}R$ is the canonical homomorphism and $\tilde{\alpha}\colon S^{-1}Q\to T^{-1}R$ is the homomorphism induced by $\alpha\colon Q\to R$.
\end{prop}
\begin{proof}
    Let $M$ be a $T^{-1}R/pT^{-1}R$-module and $D\colon R\to M$ be an FW-derivation. Then we have a well-defined map $D'\colon T^{-1}R\to M$ given by
    \begin{equation*}
        D'\left(\frac{a}{t}\right)=\frac{t^p D(a)-a^p D(t)}{t^{2p}}.
    \end{equation*}
    We show that $D'$ satisfies
    \begin{align}
        D'\left(\frac{a_1}{t_1}+\frac{a_2}{t_2}\right) &= D'\left(\frac{a_1}{t_1}\right)+D'\left(\frac{a_2}{t_2}\right)-P\left(\frac{a_1}{t_1},\frac{a_2}{t_2}\right)D'(p) \label{eq:loc-sum}, \\
        D'\left(\frac{a_1}{t_1}\cdot \frac{a_2}{t_2}\right) &= \left(\frac{a_2}{t_2}\right)^pD'\left(\frac{a_1}{t_1}\right)+\left(\frac{a_1}{t_1}\right)^pD'\left(\frac{a_2}{t_2}\right) \label{eq:loc-prod}.
    \end{align}
    The equality \eqref{eq:loc-prod} easily follows. To prove \eqref{eq:loc-sum}, we may assume $t_1=t_2=t$.
    Since $M$ is annihilated by $p$, we have
    \begin{align*}
        D'\left(\frac{a_1+a_2}{t}\right) &= \frac{t^pD(a_1+a_2)-(a_1+a_2)^pD(t)}{t^{2p}} \\
        &= \frac{t^p(D(a_1)+D(a_2)-P(a_1,a_2)D(p))-(a_1^p+a_2^p)D(t)}{t^{2p}} \\
        &= D'\left(\frac{a_1}{t}\right)+D'\left(\frac{a_2}{t}\right)-P\left(\frac{a_1}{t},\frac{a_2}{t}\right)D'(p).
    \end{align*}
    Thus $D'$ is a unique FW-derivation extending $D$ and we have the first isomorphism. We also have the second isomorphism since every homomorphism of monoids $Q\to M$ is uniquely extended to a homomorphism $S^{-1}Q\to M$ for every $T^{-1}R/pT^{-1}R$-module $M$.
\end{proof}

Let $P_n(X_1,\ldots,X_n)\in \mathbb{Z}[X_1,\ldots,X_n]$ be the polynomial given by
\begin{equation*}
    P_n(X_1,\ldots,X_n)=\sum_{\substack{0\leq i_1,\ldots,i_n<p\\i_1+\cdots+i_n=p}} \frac{(p-1)!}{i_1!\cdots i_n!}X_1^{i_1}\cdots X_n^{i_n}.
\end{equation*}
Note that we have $pP_n(X_1,\ldots,X_n)=(X_1+\cdots+X_n)^p-(X_1^p+\cdots+X_n^p)$ and $P_2(X_1,X_2)=P(X_1,X_2)$.
In particular, every FW-derivation $D\colon R\to M$ satisfies
\begin{equation*}
    D(a_1+\cdots+a_n)=D(a_1)+\cdots+D(a_n)-P_n(a_1,\ldots,a_n)D(p).
\end{equation*}

\begin{prop}\label{prop:mon-alg}
    For a ring $R$ and a monoid $Q$, let $S=R[Q]$ be the associated monoid algebra and $\iota\colon Q\to S=R[Q]$ be the natural homomorphism. Assume that $Q$ is integral. Then we have
    \begin{equation*}
        \tilde{F}\Omega_{(S,Q,\iota)}\cong \tilde{F}\Omega_R\oplus (S/pS\otimes_\mathbb{Z} Q^{gp}).
    \end{equation*}
    In particular, for an element $a$ of $R$ and an element $x$ of $Q$, the element $w(a)$ corresponds to the element $(w(a),0)$ and the element $w\log(x)$ corresponds to the element $(0,1\otimes x)$.
\end{prop}

\begin{proof}
    Let $M$ be an $S/pS$-module. It suffices to show that the homomorphism
    \begin{equation}\label{eq:4.3}
        F\Der_{(S,Q,\iota)}(M)\to F\Der_R(M)\times \Hom_\textbf{Mon}(Q,M)
    \end{equation}
    sending $(D,\delta)$ to $(D\vert_R,\delta)$ is an isomorphism. Since we have $D\iota(x)=\iota(x)^p\delta(x)$ for every $x\in Q$, the map \eqref{eq:4.3} is injective.

    We show that the map \eqref{eq:4.3} is surjective. Let $D_0\colon R\to M$ be an FW-derivation and $\delta\colon Q\to M$ be a homomorphism of monoids.
    For an element $f=a_1\iota(x_1)+\cdots+a_n\iota(x_n)\in S$ with $x_i\neq x_j$ for $i\neq j$, we set
    \begin{equation*}
        \tilde{P}(f)=P_n(a_1\iota(x_1),\ldots,a_n\iota(x_n)).
    \end{equation*}
    Define a map $D\colon S\to M$ by
    \begin{equation*}
        D\left(\sum_x a_x \iota(x)\right)=\sum_x \iota(x)^p D_0(a_x)+\sum_x a_x^p \iota(x)^p \delta(x)-\tilde{P}\left(\sum_x a_x \iota(x)\right)D_0(p).
    \end{equation*}
    Note that we have $D\vert_R=D_0$ and $D\iota(x)=\iota(x)^p\delta(x)$ for every $x\in Q$.
    Moreover, we have $D\iota(xy)=\iota(y)^p D\iota(x)+\iota(x)^p D\iota(y)$.
    It suffices to show that $D\colon S\to M$ satisfies
    \begin{align}
        D(f+g) &= D(f)+D(g)-P(f,g)D(p) \label{eq:4-sum}, \\
        D(fg) &= g^pD(f)+f^pD(g) \label{eq:4-prod}.
    \end{align}
    Write $f=\sum_x a_x \iota(x)$ and $g=\sum_x b_x \iota(x)$. Then we have
    \begin{equation*}
        \begin{split}
            D(f+g) &= \sum_x \iota(x)^p D(a_x+b_x)+\sum_x (a_x+b_x)^p D\iota(x)-\tilde{P}(f+g)D(p) \\
            &= \sum_x \iota(x)^p (D(a_x)+D(b_x)-P(a_x,b_x)D(p))+\sum_x (a_x^p+b_x^p)D\iota(x) \\
            &\quad -\left(\tilde{P}(f)+\tilde{P}(g)+P(f,g)-\sum_x P(a_x\iota(x),b_x\iota(x))\right)D(p) \\
            &= D(f)+D(g)-P(f,g)D(p).
        \end{split}
    \end{equation*}
    This proves \eqref{eq:4-sum}. Next, consider \eqref{eq:4-prod}. Let $g_0=c\iota(x_0)\in S$ be a monomial. Note that we have $x_0x\neq x_0x'$ for $x\neq x'$, since $Q$ is integral. Then we have
    \begin{align}
        D(fg_0) &= D\left(\sum_x a_xc \iota(x x_0)\right) \label{eq:4-prod-mon} \\
        &= \sum_x \iota(x x_0)^p D(a_xc)+\sum_x (a_xc)^p D\iota(x x_0)-g_0^p\tilde{P}(f)D(p) \notag \\
        &= g_0^p(D(f)+\tilde{P}(f)D(p))+f^pD(g_0)-g_0^p\tilde{P}(f)D(p) \notag \\
        &= g_0^pD(f)+f^pD(g_0). \notag
    \end{align}
    By using \eqref{eq:4-sum} and \eqref{eq:4-prod-mon}, we have
    \begin{align*}
        D(fg) &= D\left(\sum_x fb_x\iota(x)\right) \\
        &= \sum_x D(fb_x\iota(x))-f^p\tilde{P}(g)D(p) \\
        &= \sum_x (b_x\iota(x))^pD(f)+\sum_x f^p D(b_x\iota(x))-f^p\tilde{P}(g)D(p) \\
        &= g^pD(f)+f^p(D(g)+\tilde{P}(g)D(p))-f^p\tilde{P}(g)D(p) \\
        &= g^pD(f)+f^pD(g).
    \end{align*}
    This proves \eqref{eq:4-prod}.
    Therefore $D\colon S\to M$ is an FW-derivation.
\end{proof}

\section{The relation between log FW-differentials and log regularity}

\indent

In this section, we prove Main Theorem (Theorem \ref{thm:logreg-cr}).

\begin{lem}\label{lem:red-to-sharp}
    Let $(R,Q,\alpha)$ be a prelog ring such that $Q$ is integral and that $\overline{Q}$ is fine and saturated. Let $s\colon \overline{Q}\to Q$ be a section of the canonical surjection $\pi\colon Q\to \overline{Q}$. Then the homomorphism $(R,\overline{Q},\alpha s)\to (R,Q,\alpha)$ induces a natural isomorphism
    \begin{equation*}
        F\Omega_{(R,\overline{Q},\alpha s)}\cong F\Omega_{(R,Q,\alpha)}.
    \end{equation*}
\end{lem}
\begin{proof}
    It suffices to show that the homomorphism
    \begin{equation}\label{eq:red-to-sharp}
        F\Der_{(R,Q,\alpha)}(M)\to F\Der_{(R,\overline{Q},\alpha s)}(M)
    \end{equation}
    sending $(D,\delta)$ to $(D,\delta s)$ is an isomorphism for every $R$-module $M$. For every $x\in Q$, there is a unique element $u_x$ of $Q^\times$ such that $x=u_x\cdot s\pi(x)$.

    First, we show that the map \eqref{eq:red-to-sharp} is injective. Let $(D,\delta)$ be an element of $F\Omega_{(R,Q,\alpha)}(M)$ and suppose $(D,\delta s)=(0,0)$. Then, for every $x\in Q$, we have
    \begin{align*}
        \delta(x) &= \delta(u_x)+\delta(s\pi(x)) \\
        &= \delta(u_x) \\
        &= D\alpha(u_x)/\alpha(u_x)^p=0.
    \end{align*}
    Hence we have $(D,\delta)=(0,0)$ and the map \eqref{eq:red-to-sharp} is injective.
    
    Next, we show that the map \eqref{eq:red-to-sharp} is surjective.
    For an element $(D,\delta')$ of $F\Der_{(R,\overline{Q},\alpha s)}$, we define a homomorphism of monoids $\delta\colon Q\to M$ by
    \begin{equation*}
        \delta(x)=D\alpha(u_x)/\alpha(u_x)^p+\delta'(\pi(x)).
    \end{equation*}
    Then we have
    \begin{align*}
        D\alpha(x) &= D(\alpha(u_x)\cdot \alpha s\pi(x)) \\
        &= \alpha s\pi(x)^p D\alpha(u_x)+\alpha(u_x)^p D\alpha s\pi(x) \\
        &= \alpha s\pi(x)^p D\alpha(u_x)+\alpha(u_x)^p \alpha s\pi(x)^p \delta'(\pi(x)) \\
        &= \alpha(x)^p \delta(x),
    \end{align*}
    which implies that $(D,\delta)$ is an element of $F\Der_{(R,Q,\alpha)}(M)$. By the construction of $\delta$, we have $\delta'=\delta s$. Therefore the map \eqref{eq:red-to-sharp} is a surjection.
\end{proof}

\begin{lem}\label{lem:factor}
    Let $(R,Q,\alpha)$ be a prelog ring such that $Q$ is sharp. Then we have
    \begin{equation*}
        F\Omega_{(R,Q,\alpha)}\otimes_R R/I_\alpha \cong F\Omega_{(R/I_\alpha,Q,\pi\alpha)} \cong F\Omega_{R/I_\alpha}\oplus (R/I_\alpha \otimes_\mathbb{Z} Q^{gp}),
    \end{equation*}
    where $\pi$ denotes the canonical surjection $R\to R/I_\alpha$.
\end{lem}

\begin{proof}
    We have the first isomorphism by Proposition \ref{prop:logFW-basic} (2). For the second isomorphism, it suffices to show that
    \begin{equation*}
        F\Der_{(R/I_\alpha,Q,\pi\alpha)}(M)=F\Der_{R/I_\alpha}(M)\times \Hom_\mathbf{Mon}(Q,M)
    \end{equation*}
    for every $R/I_\alpha$-module $M$. Let $D\colon R/I_\alpha\to M$ be an FW-derivation and $\delta\colon Q\to M$ be a homomorphism of monoids. Then the equality
    \begin{equation*}
        D\pi\alpha(x)=\pi\alpha(x)^p\delta(x)
    \end{equation*}
    holds for every $x\in Q$. Indeed, the equality obviously holds for $x=1$. If $x\neq 1$, then we have $x\in Q^+$ and $\pi\alpha(x)=0$ since $Q$ is sharp. Thus the equality also holds for $x\neq 1$.
\end{proof}

\begin{prop}\label{prop:completion}
    Let $(R,Q,\alpha)$ be a prelog ring and $I$ be an ideal of $R$. Let $(-)^\wedge$ denote the $I$-adic completion.
    \begin{itemize}
        \item[\textup{(1)}] Let $\widehat{\alpha}$ be the composition of $Q\stackrel{\alpha}{\to}R\to \widehat{R}$. Then we have
            \begin{equation*}
                F\Omega_{(R,Q,\alpha)}^\wedge\cong F\Omega_{(\widehat{R},Q,\widehat{\alpha})}^\wedge.
            \end{equation*}
        \item[\textup{(2)}] Let $J$ be an ideal of $R$. Then we have
            \begin{equation*}
                F\Omega_{(R,Q,\alpha)}^\wedge \otimes_R R/J\cong F\Omega_{(R,Q,\alpha)}^\wedge \otimes_{\widehat{R}} \widehat{R}/J\widehat{R}\cong (F\Omega_{(R,Q,\alpha)}\otimes_R R/J)^\wedge.
            \end{equation*}
    \end{itemize}
\end{prop}
\begin{proof}
    (1) Let $M$ be an $I$-adically complete $R$-module. Then we have
    \begin{equation*}
        \Hom_R(F\Omega_{(R,Q,\alpha)}^\wedge,M)\cong \Hom_R(F\Omega_{(R,Q,\alpha)},M)\cong F\Der_{(R,Q,\alpha)}(M).
    \end{equation*}
    It suffices to show that every FW-derivation $D\colon R\to M$ is uniquely extended to an FW-derivation $\widehat{D}\colon \widehat{R}\to M$. For an element $a$ of $\widehat{R}$, take a Cauchy sequence $(a_n)_n$ in $R$ converging to $a$. Then we have
    \begin{equation*}
        D(a_m)-D(a_n)=D(a_m-a_n)-P(a_n,a_m-a_n)D(p).
    \end{equation*}
    If we have $a_m-a_n\in I^r$, then we have $D(a_m)-D(a_n)\in I^{r-1}M$.
    Hence $(D(a_n))_n$ is a Cauchy sequence in $M$. Since $M$ is complete, the sequence $(D(a_n))_n$ converges. Moreover, the element $\lim_n D(a_n)$ is independent of the choice of $(a_n)_n$ converging to $a$. Therefore, we have a homomorphism $\widehat{D}\colon \widehat{R}\to M$ sending $\lim_n a_n$ to $\lim_n D(a_n)$. It is easily checked that $\widehat{D}\colon \widehat{R}\to M$ is an FW-derivation. The uniqueness of $\widehat{D}$ directly follows from the formula $\widehat{D}(\lim_n a_n)=\lim_n D(a_n)$.

    (2) We obviously have the first isomorphism. For every $I$-adically complete $R/J$-module $M$, we have
    \begin{align*}
        & \Hom_R(F\Omega_{(R,Q,\alpha)}^\wedge\otimes_{\widehat{R}}\widehat{R}/J\widehat{R},M)\cong \Hom_R(F\Omega_{(R,Q,\alpha)},M), \\
        & \Hom_R((F\Omega_{(R,Q,\alpha)}\otimes_R R/J)^\wedge,M)\cong \Hom_R(F\Omega_{(R,Q,\alpha)},M).
    \end{align*}
    Then we have the second isomorphism.
\end{proof}

\begin{cor}\label{cor:completion}
    Let $(R,Q,\alpha)$ be a prelog ring and $I$ be an ideal of $R$. Assume that $R$ is Noetherian. Let $(-)^\wedge$ denote the $I$-adic completion.
    \begin{itemize}
        \item[\textup{(1)}] We have
            \begin{equation*}
                \tilde{F}\Omega_{(R,Q,\alpha)}^\wedge\cong \tilde{F}\Omega_{(\widehat{R},Q,\widehat{\alpha})}^\wedge,
            \end{equation*}
            where $\widehat{\alpha}$ is the composition of $Q\stackrel{\alpha}{\to}R\to \widehat{R}$.
        \item[\textup{(2)}] If $R/pR$ is F-finite and $Q$ is finitely generated, then we have
            \begin{equation*}
                \tilde{F}\Omega_{(R,Q,\alpha)}\otimes_R \widehat{R}/p\widehat{R}\cong \tilde{F}\Omega_{(\widehat{R},Q,\widehat{\alpha})}.
            \end{equation*}
    \end{itemize}
\end{cor}
\begin{proof}
    (1) The assertion follows from Lemma \ref{prop:completion}.

    (2) Let $\{x_1,\ldots,x_d\}$ be a generating set of $I$. Then we have a surjection
    \begin{equation*}
        (R/pR)[[T_1,\ldots,T_d]]\to \widehat{R}/p\widehat{R}        
    \end{equation*}
    and hence $\widehat{R}/p\widehat{R}$ is $F$-finite.
    Then the $R/pR$-module $\tilde{F}\Omega_{(R,Q,\alpha)}$ and the $\widehat{R}/p\widehat{R}$-module $\tilde{F}\Omega_{(\widehat{R},Q,\widehat{\alpha})}$ are finitely generated by Corollary \ref{prop:logFW-fin}.
    Hence the assertion follows from (1).
\end{proof}

Finally, we prove Main Theorem (Theorem \ref{thm:logreg-cr}).
For a local ring $(R,\mathfrak{m})$, let $\mathrm{edim}(R)$ denote the dimension of the $R/\mathfrak{m}$-vector space $\mathfrak{m}/\mathfrak{m}^2$.
\begin{thm}\label{thm:logreg-cr}
    Let $(R,Q,\alpha)$ be a local prelog ring such that $R$ is a Noetherian local ring with the residue field $k$ of characteristic $p>0$, that $Q$ is integral, and that $\overline{Q}$ is fine and saturated. Assume that $k$ is $F$-finite and put $[k:k^p]=p^r$. Consider the following conditions:
    \begin{itemize}
        \item[\textup{(1)}] $\tilde{F}\Omega_{(R,Q,\alpha)}$ is a free $R/pR$-module of rank $\dim(R)+r$.
        \item[\textup{(2)}] $\tilde{F}\Omega_{(R,Q,\alpha)}\otimes_R k$ is a $k$-vector space of dimension $\dim(R)+r$.
        \item[\textup{(3)}] $(R,Q,\alpha)$ is log regular.
    \end{itemize}
    Then we always have $(1)\Rightarrow (2)\Leftrightarrow (3)$. Moreover, if $R/pR$ is $F$-finite, then the three conditions $(1)$, $(2)$, and $(3)$ are equivalent.
\end{thm}

\begin{proof}
    Let $s\colon \overline{Q}\to Q$ be a section of the natural surjection $Q\to \overline{Q}$, which exists by Proposition \ref{prop:mon2}. By Proposition \ref{prop:log-reg1} (2) and Lemma \ref{lem:red-to-sharp}, we can replace $(R,Q,\alpha)$ by $(R,\overline{Q},\alpha s)$. Therefore we may assume that $Q$ is fine, saturated, and sharp.
    
    Note that the implication $(1)\Rightarrow (2)$ is obvious.
    We show that (2) is equivalent to (3). By Lemma \ref{lem:factor}, we have
    \begin{equation*}
        F\Omega_{(R,Q,\alpha)}\otimes_R k\cong (F\Omega_{R/I_\alpha}\otimes_R k)\oplus (k\otimes_\mathbb{Z} Q^{gp}).
    \end{equation*}
    Note that the dimension of $F\Omega_{R/I_\alpha}\otimes_R k$ is equal to $\mathrm{edim}(R/I_\alpha)+r$ by \cite[Proposition 2.6]{MR4412577}. By Proposition \ref{prop:mon1}, we have $\dim_k(k\otimes_\mathbb{Z} Q^{gp})=\mathrm{rank}(Q^{gp})=\dim(Q)$ since $Q$ is fine, saturated, and sharp. Thus we have the inequality
    \begin{align*}
        \dim(R/I_\alpha) +\dim(Q)&\leq \mathrm{edim}(R/I_\alpha)+\dim(Q) \\
        &= \dim_k(F\Omega_{(R,Q,\alpha)}\otimes_R k)-r.
    \end{align*}
    On the other hand, we have the inequality $\dim(R)\leq \dim(R/I_\alpha)+\dim(Q)$ by Proposition \ref{prop:log-reg1} (1).
    Therefore, we have $(2)\Leftrightarrow (3)$.

    Finally, we show that (2) and (3) imply (1) under the assumption that $R/pR$ is $F$-finite.
    We only need to show that the $R/pR$-module $\tilde{F}\Omega_{(R,Q,\alpha)}$ is free.
    Since $R/pR\to \widehat{R}/p\widehat{R}$ is faithfully flat and the $R/pR$-module $\tilde{F}\Omega_{(R,Q,\alpha)}$ is finitely generated by Proposition \ref{prop:logFW-fin}, the $R/pR$-module $\tilde{F}\Omega_{(R,Q,\alpha)}$ is free if and only if the $\widehat{R}/p\widehat{R}$-module $\tilde{F}\Omega_{(R,Q,\alpha)}$ is free.
    Thus we may assume $R=\widehat{R}$ by Corollary \ref{cor:completion} (2). Put $n=\dim(R/I_\alpha)$.

    \vspace{5pt}
    \textbf{(\Rnum{1}) The positive characteristic case}: Suppose that $R$ is of characteristic $p$. Put $S_0=k[T_1,\ldots,T_n]$ and $S=S_0[Q]$. Let $\iota\colon Q\to S$ be the natural homomorphism. By Theorem \ref{thm:logreg-str}, we have $(R,Q,\alpha)\cong (\widehat{S},Q,\widehat{\iota})$, where $\widehat{S}$ denotes the completion of $S$ with respest to the maximal ideal $\mathfrak{m}_S$ generated by $T_1,\ldots,T_n$ and $\iota(Q^+)$. Note that $S$ is $F$-finite.
    It suffices to show that the $S$-module $F\Omega_{(S,Q,\iota)}$ is free since we have $F\Omega_{(\widehat{S},Q,\widehat{\iota})}\cong F\Omega_{(S,Q,\iota)}^\wedge$ by Corollary \ref{cor:completion} (2).
    By Proposition \ref{prop:mon-alg}, we have
    \begin{equation*}
        F\Omega_{(S,Q,\iota)}\cong (F\Omega_{S_0}\otimes_{S_0}S)\oplus (S\otimes_\mathbb{Z} Q^{gp}).
    \end{equation*}
    The $S_0$-module $F\Omega_{S_0}$ is free since we have $F\Omega_{S_0}=F^*\Omega_{S_0}$. Therefore, the $S$-module $F\Omega_{(S,Q,\iota)}$ is free.

    \vspace{5pt}
    \textbf{(\Rnum{2}) The mixed characteristic case}: Suppose that $R$ is of mixed characteristic. Let $V$ be a Cohen ring of $k$. Put $S_0=V[T_1,\ldots,T_n]$ and $S=S_0[Q]$. Let $\iota\colon Q\to S$ be the natural homomorphism. Let $\widehat{S}$ denote the completion of $S$ with respect to the maximal ideal $\mathfrak{m}_S$ generated by $p,T_1,\ldots,T_n$ and $\iota(Q^+)$. By Theorem \ref{thm:logreg-str}, there is an element $f$ of $\widehat{S}$ satisfying
    \begin{itemize}
        \item The constant term of $f$ is $p$,
        \item There is an isomorphism $(R,Q,\alpha)\cong (\widehat{S}/f\widehat{S},Q,\pi\widehat{\iota})$, where $\pi$ is the canonical surjection $\widehat{S}\to \widehat{S}/f\widehat{S}$.
    \end{itemize}
    Note that $S/pS\cong k[T_1,\ldots,T_n][Q]$ is $F$-finite.
    By Proposition \ref{prop:logFW-basic} (2), we have
    \begin{equation*}
        \tilde{F}\Omega_{(\widehat{S}/f\widehat{S},Q,\pi\widehat{\iota})}\cong \frac{\tilde{F}\Omega_{(\widehat{S},Q,\widehat{\iota})}\otimes_{\widehat{S}}\widehat{S}/(p,f)}{\langle w(f)\rangle}.
    \end{equation*}
    By Corollary \ref{cor:completion} (2), we have
    \begin{equation*}
        \tilde{F}\Omega_{(\widehat{S},Q,\widehat{\iota})}\cong \tilde{F}\Omega_{(S,Q,\iota)}\otimes_S \widehat{S}/p\widehat{S}.
    \end{equation*}
    By Proposition \ref{prop:mon-alg} and \cite[Proposition 2.5 (3)]{MR4412577}, we have
    \begin{align*}
        \tilde{F}\Omega_{(S,Q,\iota)} &\cong (F\Omega_{S_0}\otimes_{S_0} S/pS)\oplus (S/pS\otimes_\mathbb{Z} Q^{gp}) \\
        &\cong (F\Omega_V\otimes_V S/pS)\oplus (S/pS)^n\oplus (S/pS\otimes_\mathbb{Z} Q^{gp}).
    \end{align*}
    By \cite[Proposition 2.6]{MR4412577}, the element $w(p)$ of the $k$-vector space $F\Omega_V$ belongs to a linear basis of $F\Omega_V$.
    Hence the $S/pS$-module $\tilde{F}\Omega_{(S,Q,\iota)}$ is free and there are elements $v_1,\ldots,v_l$ of $F\Omega_{(S,Q,\iota)}$ such that
    \begin{equation*}
        \{w(p),w(T_1),\ldots,w(T_n),v_1,\ldots,v_l\}\subset \tilde{F}\Omega_{(S,Q,\iota)}
    \end{equation*}
    is a basis of $\tilde{F}\Omega_{(S,Q,\iota)}$.
    In particular, the $\widehat{S}/(p,f)$-module $\tilde{F}\Omega_{(\widehat{S},Q,\widehat{\iota})}\otimes_{\widehat{S}}\widehat{S}/(p,f)$ is free and
    \begin{equation*}
        \{w(p),w(T_1),\ldots,w(T_n),v_1,\ldots,v_l\}\subset (\tilde{F}\Omega_{(\hat{S},Q,\hat{\iota})}\otimes_{\hat{S}}\hat{S}/(p,f))^\wedge
    \end{equation*}
    is a basis of $\tilde{F}\Omega_{(\widehat{S},Q,\widehat{\iota})}\otimes_{\widehat{S}}\widehat{S}/(p,f)$.
    
    We compute $w(f)$. Let $\{x_1,\ldots,x_m\}\subset Q^+$ be a generating set of $Q$. Since the constant term of $f$ is $p$, we can write
    \begin{equation*}
            f-p=g_1 \widehat{\iota}(x_1)+\cdots+g_m \widehat{\iota}(x_m)+h_1 T_1+\cdots+h_n T_n,
    \end{equation*}
    where $g_i,h_i\in \hat{S}$. Then we have
    \begin{equation*}
        \begin{split}
            w(f) &\equiv w(p)+w(f-p)\mod \widehat{\mathfrak{m}}_S \\
            &\equiv w(p)+\sum_i w(g_i \widehat{\iota}(x_i))+\sum_j w(h_j T_j)\mod \widehat{\mathfrak{m}}_S \\
            &\equiv w(p)+\sum_i \widehat{\iota}(x_i)^p w(g_i)+\sum_i g_i^p\widehat{\iota}(x_i)^p w\log(x_i) \\
            &\quad +\sum_j T_j^p w(h_j)+\sum_j h_j^p w(T_j)\mod \widehat{\mathfrak{m}}_S \\
            &\equiv w(p)+\sum_j h_j^p w(T_j)\mod \widehat{\mathfrak{m}}_S.
        \end{split}
    \end{equation*}
    Therefore, we can write
    \begin{equation*}
        w(f)=\varepsilon w(p)+\sum_i a_i w(T_i)+\sum_j b_j v_j,
    \end{equation*}
    where $\varepsilon\in\widehat{S}^\times$ and $a_i,b_j\in \widehat{S}$.
    Hence the $\widehat{S}/(p,f)$-module
    \begin{equation*}
        \tilde{F}\Omega_{(\widehat{S}/f\widehat{S},Q,\pi\widehat{\iota})}\cong \frac{\tilde{F}\Omega_{(\widehat{S},Q,\widehat{\iota})}\otimes_{\widehat{S}}\widehat{S}/(p,f)}{\langle w(f)\rangle}
    \end{equation*}
    is free.
\end{proof}

\begin{rem}
    By taking $Q=1$ in Theorem \ref{thm:logreg-cr}, we obtain a criterion for usual regularity which is valid under a weaker assumption than that of Theorem \ref{thm:Saito_reg} (\cite[Theorem 3.4]{MR4412577}) and that of Theorem \ref{thm:HJ_reg} (\cite[Theorem 4.13]{MR4855866}).
    
    Indeed, if $R/\sqrt{pR}$ is essentially of finite type over an $F$-finite field $k_1$, then $R/\sqrt{pR}$ and hence $R/pR$ are $F$-finite.
    
    If $R$ is essentially of finite type over $V[[T_1,\ldots,T_n]]$, where $V$ is a discrete valuation ring with uniformizer $p$ such that $V/pV$ is $F$-finite, then $R/pR$ is essentially of finite type over $(V/pV)[[T_1,\ldots,T_n]]$ and hence $F$-finite.
\end{rem}


\end{document}